\documentclass{amsart}

\usepackage{amssymb}
\usepackage{amscd}
\usepackage{graphicx}
\usepackage{amsmath}
\usepackage{setspace}
\usepackage[all]{xy}
\usepackage{color}
\usepackage{enumerate}

\newtheorem{theorem}{Theorem}[section]

\newtheorem{corollary}[theorem]{Corollary}

\newtheorem{lemma}[theorem]{Lemma}

\newtheorem{proposition}[theorem]{Proposition}

\theoremstyle{remark}

\newtheorem{remark}[theorem]{Remark}

\newcommand{\CC}{\mathcal{C}}

\newcommand{\CG}{\mathcal{G}}

\newcommand{\CS}{\mathcal{S}}

\newcommand{\CV}{\mathcal{V}}

\newcommand{\CW}{\mathcal{W}}

\newcommand{\rmW}{\mathrm{W}}

\newcommand{\BBP}{{\mathbb P}}
\newcommand{\BBQ}{{\mathbb Q}}

\newcommand{\BBZ}{{\mathbb Z}}

\newcommand{\Hom}{\mathrm{Hom}}
\newcommand{\Ima}{\mathrm{Im}}

\newcommand{\Ker}{\mathrm{Ker}}

\newcommand{\Ext}{\mathrm{Ext}}

\newcommand{\End}{\mathrm{End}}

\title[On a theorem of Stelzer]{On a theorem of Stelzer for some classes of mixed groups}

\author{Simion Breaz}
\address[bodo@math.ubbcluj.ro]{"Babe\c s-Bolyai" University, Faculty of Mathematics and Computer Science, Str. Mihail Kog\u alniceanu 1, 400084, Cluj-Napoca, Romania}

\begin{document}

\begin{abstract}
We identify some classes $\mathcal{C}$ of mixed groups such that if $G\in \mathcal{C}$ has the cancellation property then the Walk-endomorphism ring of $G$ has the unit lifting property. In particular,  if $G$ is a self-small group of torsion-free rank at most $4$ with the cancellation property then it has a decomposition $G=F\oplus H$ such that $F$ is free and the Walk-endomorphism ring of $H$ has the unit lifting property.
\end{abstract}

%\begin{summary} % This is required for articles that are not in English
% Let us write some maths.
%\end{summary}


\maketitle

\section{Introduction}
In this paper all groups are abelian. If $G$ is a group then $T(G)$ represents the torsion part of $G$, for a prime $p$ the $p$-component of $G$ is denoted by $T_p(G)$, $r_0(G)$ is the torsion-free rank of $G$, and $\End(G)$ is the endomorphism ring of $G$. For other notions and notations we refer to \cite{Fuchs-2015} and to Section \ref{prel-exist-H}. 

A group $G$ has the \textit{cancellation property} if whenever $H$ and $K$ are  groups such that $G\oplus H\cong G\oplus K$ it follows that $H\cong K$. An important question is if we can find classes of groups which have the cancellation property and they can be characterized by using properties of some endomorphism rings. For instance a group $G$ has the substitution property if and only if one is in the stable range of $\End(G)$, \cite{Warf}.

For finite rank torsion-free groups the cancellation property was characterized in \cite[Theorem 21]{Bla} by using properties of some endomorphism rings: a torsion-free group $G$ of finite rank has the cancellation property if and only if $G=B\oplus C$ such that (a) $B$ is free (b) for every positive integer $n$ the units of $\End(C)/n\End(C)$ can be lifted to units of $\End(C)$, and  (c) the endomorphism rings of all quasi-direct summands of $C$ satisfy some technical conditions discovered by Eichler in the 1930s. If a ring $R$ has the property described in (b), i.e. for every positive integer $n$ all units of $R/nR$ can be lifted to units of $R$, then we will say that $R$ has \textit{the unit lifting property}. The properties (a) and (b) for necessity part of Blazhenov's characterization were proved in \cite[Theorem A]{Stel}.  
In fact, Stelzer proved that if $G$ is a reduced torsion-free group of finite rank without free direct summands there exists a torsion-free group $H$ such that (i) $\Hom(H,G)=0=\Hom(G,A)$, (ii) $\End(H)\cong \BBZ$, and (iii) for every positive integer $n$ there is an epimorphism $A\to G/nG$.
Then the conclusion comes from a result of Fuchs, \cite{Fuchs-canc}, which states that if $G$ satisfies these three properties and it has the cancellation property then $\End(G)$ has the unit lifting property.
  
We denote by $\mathcal{S}$ the class of self-small (mixed) groups of finite torsion-free rank. Some of results about finite rank torsion-free groups and also some techniques used in the theory of finite rank torsion-free groups can be extended to some subclasses of $\CS$, e.g. \cite{ABVW}, \cite{Br-qd},  \cite{Br-Warf}. However, this is not alway valid. We refer to \cite[Example 2.5]{ABVW} and \cite[Example 3.6]{Br-Warf} for some examples. 
Regarding the results about the cancellation property, we note that in \cite[Section 2]{ABVW} the connections between the substitution property and the unit lifting property proved in \cite{Ar} for torsion-free groups were extended for groups in $\CS$.  Mixed versions of \cite[Theorem B]{Stel} were proved for the class $\mathcal{QD}$ of quotient-divisible groups (we remind that $\mathcal{QD}\subseteq \CS$) and for self-small groups of torsion-free rank $1$ in \cite[Theorem 3.4 and Proposition 3.5]{ABVW}. A mixed version of Fuchs's Lemma was proved in 
\cite{Br-2020}, see Proposition \ref{fuchs-mixed}.

One of the aims of the present paper is to prove that a mixed version for
Stelzer's Theorem \cite[Theorem A]{Stel}, which uses the Walk-endomorphism ring instead of the classical endomorphism ring, is valid for some classes of self-small groups. In particular, if a self-small group is of torsion-free rank at most $4$ has no free direct summands and it has the cancellation property then its Walk-endomorphism ring has the unit lifting property (Theorem \ref{main-thm}).  %if $G\in \CS$ is of the torsion-free rank at most $4$ and it has the cancellation property then $G=F\oplus H$ such that $F$ is free and the Walk-endomorphism ring of $H$, $\End_{\rmW}(H)=\End(H)/\Hom(H,T(H))$, has the unit lifting property. 

For reader's convenience, we close this introduction with a sketch of the proof of this result. It is easy to see that every self-small group $G$ of finite torsion-free rank has a direct decomposition $G=F\oplus H$ such that $F$ is free and $H$ is self-small without free direct summands. Moreover, $G$ has the cancellation property if and only if $H$ has the cancellation property. Hence it is enough to study self-small groups without free direct summands. Similar reasons allow us to restrict ourselves to the case of reduced groups. 
\textit{Let $\CS^\star$ be the class of reduced self-small groups of finite torsion-free rank without free direct summands}. Our aim is to apply the mixed version of Fuchs' lemma, Proposition \ref{fuchs-mixed}. Therefore, in Section \ref{sufficient-conditions} we will provide some sufficient conditions which can be verified in order to conclude that a group $G\in \CS^\star$ satisfies the hypothesis of this proposition. In the next section we will write every group $G\in\CS^\star$ as a sum $G=L+K$ where $L\cap K$ is a full free subgroup of $G$, $L$ is quotient-divisible and $K$ is torsion-free such that the Richman type of $K$ is reduced. Since all groups used here are of finite torsion-free rank, there exists a decomposition $L=F_1\oplus L_1$ such that $F_1$ is free and $L_1$ has no free direct summands. In Proposition \ref{cazuri} we will describe the structure of the groups $G\in\CS^\star$ for which the rank of $F_1$ or the torsion-free rank of $L_1$ is small enough. We will use this result, together with Lemma \ref{fara-calG}, to prove that if $G\in \CS^\star$ is of torsion-free rank at most $4$ and it has the cancellation property then the Walk-endomorphism ring of $G$, $\End_{\rmW}(G)=\End(G)/\Hom(G,T(G))$, has the unit lifting property, Theorem \ref{main-thm}. This will be enough to conclude that for every $G\in \CS$ of the torsion-free rank at most $4$, if $G$ has the cancellation property then $G=F\oplus H$ such that $F$ is free and the Walk-endomorphism ring $\End_{\rmW}(H)$ has the unit lifting property.

\section{Preliminaries and known results}\label{prel-exist-H}

\subsection{Self-small groups}\label{self-small-subsect}
Self-small groups were introduced by Arnold and Murley in \cite{Ar-Mu} as those groups $G$ such that for every set $I$ the canonical homomorphism $\Hom(G,G)^{(I)}\to \Hom(G,G^{(I)})$ is an isomorphism. The class of self-small groups of finite torsion-free rank is denoted by $\CS$. The following characterizations for the groups from $\CS$ are presented in \cite[Theorem 2.1]{ABW-09}. For other characterizations we refer to \cite[Section 3]{ABW-09} and \cite[Theorem 3.1]{Br-S}. 

\begin{theorem}\label{char-ss}
Let $G$ be a group of finite torsion-free rank. The following are equivalent:
\begin{enumerate}[{\rm 1)}]
 \item $G\in \CS$;
 \item for all primes $p$ the $p$-components $T_p(G)$ are finite, and $\Hom(G,T(G))$ is a torsion group;
 \item for every prime $p$ the $p$-component $T_p(G)$ is finite, and if $F_G\leq G$ is a full free subgroup of $G$ then $G/F_G$ is $p$-divisible for almost all primes $p$ such that $T_{p}(G)\neq 0$.
%\item $G=B\oplus X$ with $B$ a finite group, and there is an exact sequence $0\to Y\to X\to Z\to 0$ such that $Y$ is a torsion-free group which is $p$-divisible for all primes $p$ such that $T_p(G)\neq 0$ and $Z$ has no subgroups which are $p$-divisible for all primes $p$ with $T_p(G)\neq 0$.   
 \end{enumerate}
\end{theorem}

Consequently, if $G\in \CS$ then for every finite set $P$ of primes the $P$-component $\oplus_{p\in P}T_p(G)$ is a finite direct summand of $G$. In this case we will fix a direct decomposition $G=(\oplus_{p\in P}T_p(G))\oplus G(P)$. If $n>1$ is an integer, we denote by $P_n$ the set of all prime divisors of $n$, and we will write $G=T_n(G)\oplus G(n)$, where $T_n(G)=\oplus_{p\in P_n}T_p(G)$ and $G(n)=G(P_n)$. 

There are two important classes of groups which are contained in $\CS$: the class $\mathcal{TF}$ of finite rank torsion-free groups, \cite{Ar-Mu}, and the class $\mathcal{QD}$ of quotient-divisible groups, \cite{FoW98}. 
We recall that a group of finite torsion-free rank $G$  is \textit{quotient-divisible} if its torsion part is reduced and there exists a full free subgroup $F\leq G$ such that $G/F$ is divisible. 

If $\CC\subseteq \CS$, we will denote by $\BBQ\CC$ the \textit{quasi-category} associated to $\CC$. This is a category which have as objects the groups from $\CC$, and $\Hom_{\BBQ\CC}(G,H)=\BBQ\otimes \Hom(G,H)$ (this group is denoted by $\BBQ\Hom(G,H)$) for all $G,H\in \CC$. The reader can find a general discussion about quasi-categories in \cite{Br-Mo}. Since for every $G\in \CS$ the support group of $\End(G)$ is of finite torsion-free rank, the ring $\BBQ\End(G)$ is a finite dimensional $\BBQ$-algebra, and it follows that the the quasi-category $\BBQ\CS$ is Krull-Schmidt, \cite{Br-qd}. Moreover, a group $G\in\CS$ is indecomposable in $\BBQ\CS$ (such a group is called \textit{strongly indecomposable}) if and only if $\BBQ\End(G)$ is a local ring. 

It was proved in \cite{FoW98} that there exists a duality $d:\BBQ\mathcal{TF}\rightleftarrows \BBQ\mathcal{QD}:d$. Moreover, the restriction of $d$ to the subcategory $\BBQ\mathcal{TFQD}$ of all torsion-free groups which are quotient-divisible induce a duality $d:\BBQ\mathcal{TFQD}\to \BBQ\mathcal{TFQD}$, \cite{Ar72}. 
In the following we recall some properties associated to this duality. If $X$ is in $\mathcal{TF}$ or in $\mathcal{QD}$ then $d(X)$ is constructed in \cite[Theorem 10]{FoW98} by using a full free subgroup $F\leq X$. Therefore $d(X)$ is unique up to a quasi-isomorphism.

\begin{lemma}\label{prop-duality}
The following are true for a group $X$ and a full free subgroup $F\leq X$:
\begin{enumerate}[{\rm a)}]
 \item if $X\in \mathcal{TF}$ and $p$ is a prime, the $p$-component of $X/F$ is bounded if and only if $d(X)=B\oplus Y$ with $B$ a finite group and $Y$ a $p$-divisible group;
 \item if $X\in \mathcal{QD}$, $X/F$ is divisible and $p$ is a prime, we have $(X/F)_p=0$ if and only if $d(X)$ is $p$-divisible;
 \item $d$ preserves the torsion-free rank.
\end{enumerate}
\end{lemma}

\begin{remark}\label{rem-quasi-exact}
The duality $d$ preserves the quasi-exact sequences (i.e. exact sequences in $\BBQ\mathcal{TF}$, respectively in $\BBQ\mathcal{QD}$), \cite{Fo09}. If $0\to K\overset{f}\to L\overset{g}\to M\to 0$ is a short exact sequence of torsion-free (resp. quotient-divisible) groups then, for some suitable positive integers $k$ and $\ell$, the quasi-homomorphisms $kd(f)$ and $\ell d(g)$ can be viewed as group homomorphisms $kd(f):d(L)\to d(K)$ and $\ell d(g):d(M)\to d(L)$ such that $kd(f)\ell d(g)=0$, cf. the proof of \cite[Theorem 10]{FoW98}. It follows that $\Ker(\ell d(g))$ is finite, $\Ima(\ell d(g))$ is a finite index subgroup of $\Ker(kd(f))$, and $\Ima(k d(f))$ is a finite index subgroup of $d(K)$.  
\end{remark}

\subsection{The cancellation property} 

A group $G$ has the \textit{cancellation property} if whenever $H$ and $K$ are  groups such that $G\oplus H\cong G\oplus K$ it follows that $H\cong K$. The group $G$ has the \textit{substitution property} if for every group $A$ which has direct decompositions $A=G_1\oplus H=G_2\oplus K$ such that $G_1\cong G_2\cong G$ there exists $G_0\leq A$ such that $G_0\cong G$ and  $A=G_0\oplus H=G_0\oplus K$. It is easy to see that every group with the substitution property has the cancellation property. We refer to \cite{Fuchs-2015} for a survey on these properties. In particular, it was proved that there are strong connections between these properties and properties of the endomorphism ring of $G$: the group $G$ has the substitution property if and only in $1$ is in the stable range of $\End(G)$, \cite{Warf}; if $G\in\mathcal{TF}$ has no free direct summands and it has the cancellation property then $\End(G)$ has the unit lifting property, \cite{Stel}; if $R$ is a ring which is torsion-free of finite rank such that $\BBQ R$ is local then $R$ has the unit lifting property if and only if $1$ is in the stable range of $R$, \cite{Ar}. 

For the case of (mixed) self-small groups it was proved in \cite{ABVW} that it can be useful to use \textit{the Walk-endomorphism ring} $\End_\mathrm{W}(G)=\End(G)/\Hom(G,T(G))$ instead of the classical endomorphism ring. A group $G\in \CS$ has the substitution property if and only if $1$ is in the stable range of $\End_\mathrm{W}(G)$, \cite[Theorem 2.1]{ABVW}. Moreover, if $G\in \mathcal{QD}$ has no free direct summands and it has the cancellation property then $\End_\mathrm{W}(G)$ has the unit lifting property, \cite[Proposition 3.4]{Br-2020}.

The proofs for the unit lifting properties are bases on a result which was discovered by L. Fuchs in \cite{Fuchs-canc} for torsion-free groups. 
 We present here a mixed version of this result, proved in \cite[Proposition 3.1]{Br-2020}. An epimorphism $\alpha:H\to U$ is called \textit{rigid} if for every commutative diagram 
$$\xymatrix{ H \ar[r]^{\alpha}\ar[d]^\psi & U\ar[d]^\phi \\
H\ar[r]^\alpha & U,
}$$ 
with $\psi$ and $\phi$ isomorphisms, we have $\phi=\pm 1_U$. 

\begin{proposition}\label{fuchs-mixed} {\rm \cite[Proposition 3.1]{Br-2020}}
Suppose that $G\in\CS$, and that for every positive integer $n$ there exists a torsion-free group $H$ such that the pair $(G,H)$ has the following properties 
\begin{itemize}
 \item[$\mathrm{(I)}_{\ \ }$] $\Hom(G,H)=0$,  $\Hom(H,G)$ is a torsion group, and 
 \item[$\mathrm{(II)}_n$] there exists a rigid epimorphism $\alpha:H\to G(n)/nG(n)$.
\end{itemize}

If $G$ has the cancellation property then the Walk-endomorphism ring $\End_{\rmW}(G)=\End(G)/\Hom(G,T(G))$ has the unit lifting property.
\end{proposition}

\section{Groups which verify the hypothesis of Proposition \ref{fuchs-mixed}}\label{sufficient-conditions}

In this section we will identify some classes of self-small groups which satisfy the hypothesis of Proposition \ref{fuchs-mixed}. 

\subsection{The finite rank torsion-free case} In the case of torsion-free groups, Stelzer proved in \cite[Theorem]{Stel} that every reduced finite rank torsion-free group without free direct summands satisfies the hypothesis of Proposition \ref{fuchs-mixed}. For further use we also record other results obtained in \cite{Stel}. We recall that if $P$ is a set of primes, a group $K$ is \textit{$P$-divisible} if for all $p\in P$ and all $a\in K$ the equation $px=a$ has a solution in $K$.   

\begin{proposition}{\rm \cite{Stel}}\label{Stelzer}
Let $G\in \CS^\star$ be torsion-free of rank $k$.  
There exists a torsion-free group $H$ of rank $k+1$ such that:
 \begin{enumerate}[{\rm (i)}]
 \item $\End(H)\cong \BBZ$;
 \item there exist a finite set of primes $P$ and a free subgroup $F_H\leq H$ such that the only subgroup of $G$ which is $P$-divisible is $0$ and $H/F_H\cong \oplus_{p\in P}\BBZ(p^\infty)$;
 \item for each positive integer $n$ there is an epimorphism $H\to (\BBZ/n\BBZ)^k$;
 \item all non-zero proper subgroup of $H$ are free;
 \item $\Hom(H,G)=0$ and $\Hom(G,H)=0$.
\end{enumerate}
%\end{enumerate}

Consequently, there exists a group $H$ such that for every positive integer $n$, the pair $(G,H)$ satisfy the conditions $\mathrm{(I)}$ and $(\mathrm{II})_n$.
\end{proposition}

The statement (ii) is extracted from the proof of \cite[Theorem]{Stel}, while (iii) is a consequence of \cite[Lemma 2]{Stel}. 

\begin{remark}
One of the main ingredients used in the proof of \cite[Lemma 2]{Stel} is that for every reduced torsion-free group of finite rank $G$ there exists a finite set of primes $P$ such that $G$ has no non-trivial $P$-divisible subgroups. This property is no longer true for the mixed case. For instance, if we consider a group $G$ with the properties 
\begin{itemize}
\item[i)]$\oplus_{p\in\BBP}\BBZ/p\BBZ\leq G\leq \prod_{p\in\BBP}\BBZ/p\BBZ$, 
\item[ii)] $(\widehat{1}_p)_{p\in\BBP}$ (here $\widehat{1}_p$ denotes the unit of $\BBZ/p\BBZ$), 
\item[iii)] the factor group $G/(\oplus_{p\in\BBP}\BBZ/p\BBZ)$ is a quotient-divisible group which is not reduced,
\end{itemize}
then $\Hom(G,T(G))$ is a torsion-group, hence $G$ is self-small. However, for every finite set $P\subseteq \BBP$ the subgroup $H\leq G(P)$, defined by the properties $T(G(P))\leq H\leq G(P)$ and $H/T(G(P))$ is the divisible part of $G(P)/T(G(P))\cong G/T(G)$, is a non-trivial $P$ divisible subgroup of $G$.
\end{remark}

\subsection{The case $G/T(G)$ is reduced} 

We can apply Stelzer's existence result to groups such that $G/T(G)$ is reduced.

\begin{proposition}\label{G-bar-reduced}
If $G\in\CS^\star$ is a group such that $\overline{G}=G/T(G)$ is reduced then there exists a group $H$ such that the pair satisfies the conditions $\mathrm{(I)}$ and $(\mathrm{II})_n$ for every positive integer $n$.
%
%Consequently, if $G$ has the cancellation property then $\End_{\mathrm{W}}(G)$ has the unit lifting property.
\end{proposition}

\begin{proof}
Suppose that $r_0(G)=k$. Let $H$ be the group constructed via Proposition \ref{Stelzer} such that $\Hom(H,\overline{G})=0$ and $\Hom(\overline{G},H)=0$. 

Since for every positive integer $n$ we have $G(n)/nG(n)\cong \overline{G}/n\overline{G}$ the pair $(G,H)$ satisfies the condition $\mathrm{(II)}_n$.

The rank of $H$ is $k+1$ and all proper non-zero subgroups of $H$ are free. Then $\Hom(G,H)=0$. If $f:H\to G$ is a homomorphism and $\pi_{T(G)}:G\to \overline{G}$ is the canonical surjection then $\pi_{T(G)}f=0$. Then $\Ima(f)$ is a torsion group. Since $\Ima(f)$ is also reduced, it follows by using the condition (ii) from Proposition \ref{Stelzer} that $\Ima(f)$ is finite.  Then $(G,H)$ also satisfies the condition $(\mathrm{I})$.     
\end{proof}

\subsection{A sufficient condition}
If $G/T(G)$ is not reduced, it is more dificult to construct a group $H$ such that the pair $(G,H)$ verifies the conditions $\mathrm{(I)}$ and $(\mathrm{II})_n$. In the case of quotient-divisible groups this dificulty was passed by using an infinite family of groups. The proof of the following result can be extracted from \cite[Proposition 3.4]{Br-2020}. For reader's convenience we give here some details.

\begin{lemma}\label{suff}
Suppose that for a fixed group $G\in \CS^\star$ there exists an infinite family $\CW$ such that for all $W\in \CW$ we have 
\begin{enumerate}[{\rm a)}]
\item $\Hom(G,W)=0$;
 \item $\Hom(W,G)$ is a torsion group;
   \item $\End(W)\cong \BBZ$;
 \item if $W_1,W_2\in \CW$ and $W_1\neq W_2$ then $\Hom(W_1,W_2)=0$.
 \end{enumerate}

Then, for every positive integer $n$ there exists $H$ such that the pair $(G,H)$ verifies the conditions $\mathrm{(I)}$ and $(\mathrm{II})_n$. 
\end{lemma}

\begin{proof}
Let $n$ be positive integer. We consider a direct decomposition $G(n)/nG(n)=\oplus_{i=1}^t\langle u_i\rangle$. 
For every $i\in\{1,\dots,t\}$ and for every $W\in \CW$ there is an epimorphism $W\to \langle u_i\rangle$. 
Therefore, we can fix some epimorphisms $\alpha_{i}:W_i\to \langle u_i\rangle$, for all $i\in\{1,\dots,t\}$ and $\alpha_{j\ell}:W_{j\ell}\to \langle u_j+u_\ell\rangle$ for all $j,\ell\in\{1,\dots,t\}$ with $j<\ell$ such that the groups
$W_1,\dots, W_t,W_{12},\dots ,W_{(t-1)t}\in \CW$ are pairwise non-isomorphic.
We denote $H=\left(\oplus_{i=1}^{t}W_i\right)\oplus\left(\oplus_{1\leq j<\ell\leq t} W_{j\ell}\right)$, and we consider the epimorphism $\alpha:H\to G(n)/nG(n)$  induced by $\alpha_i$ and $\alpha_{j\ell}$. As in the proof of \cite[Proposition 3.4]{Br-2020}, it can be proved by some direct computations that $\alpha$ is rigid, hence $H$ verifies $(\mathrm{II})_n$. From a) and b) it is clear that $H$ verifies $\mathrm{(I)}$.  
\end{proof}

The following result, which is useful in the study of the case when $G/T(G)$ is not reduced, was proved in  {\rm \cite[Proposition 3.3]{ABVW}}, for the particular case $m=k+1$. In the proof of this result the groups from the class $\CV$ are constructed in \cite[Lemma 4.1]{Go-W}. It is easy to see that all details from this proof can be easily adapted for all integers $m>k$.  

\begin{proposition}\label{qd-H}
Let $L\in \CS^\star$ be a quotient-divisible with $r_0(G)=k\geq 1$. If $m>k$, there exists an uncountably family $\CW$ of  torsion-free (quotient-divisible) groups of rank $m$ such that for all $W\in \CW$ we have 
\begin{enumerate}[{\rm (i)}]
 \item $\End(W)\cong \BBZ$;
 \item $\Hom(L,W)=0$;
 \item all proper torsion-free quotient groups of $W$ are divisible; 
 \item $\Hom(W,L)$ is a torsion group;
 \item all rank $1$ subgroups of $W$ are free and there exists a free subgroup $F_W\leq W$ such that $W/F_W\cong (\BBQ/\BBZ)^{m-1}$. %\cite[Theorem 3.1]{Go-W} 
 \end{enumerate}
Moreover, if $W_1,W_2\in \CW$ and $W_1\neq W_2$ then $\Hom(W_1,W_2)=0$.
 \end{proposition}

\begin{corollary}\cite[Proposition 3.4]{Br-2020}
If $G\in \CS^\star$ is quotient-divisible then for every positive integer $n$ there exists $H$ such that  $\mathrm{(I)}$ and $(\mathrm{II})_n$ are true for the pair $(G,H)$. 
\end{corollary}

We will use the family $\CW$ from Proposition  \ref{qd-H} to prove a modified version for Proposition \ref{Stelzer}.  

\begin{proposition}\label{tf-W}
Let $K\in\CS^\star$ be torsion-free of rank $k$. Then for every $m>k$ there exists an uncountably family $\CV$ of  torsion-free quotient-divisible groups of rank $m$ such that for all $V\in \CV$ we have 
\begin{enumerate}[{\rm (i)}]
 \item $\End(V)\cong \BBZ$;
 \item $\Hom(V,K)=0$;
 \item all proper pure groups of $V$ are free (hence $\Hom(K,V)=0$);
 \item all rank $1$ torsion-free quotients of $V$ are divisible.
 \end{enumerate}
Moreover, if $V_1,V_2\in \CV$ and $V_1\neq V_2$ then $\Hom(V_1,V_2)=0$.     
\end{proposition}

\begin{proof} Let $d:\BBQ\mathcal{TF}\rightleftarrows \BBQ\mathcal{QD}:d$ be the duality presented in Section \ref{self-small-subsect}.
Since $d(\BBZ)=\BBQ$ and $d(\BBQ)=\BBZ$, it is easy to observe that the quotient-divisible group $d(K)$ belongs to $\CS^\star$. Then there exists a family $\CW$ as in Proposition \ref{qd-H} which corresponds to $d(K)$. Let $\CV=\{d(W)\mid W\in \CW\}$. Since the restriction of $d$ to the class of quotient-divisible groups coincides with Arnold's duality described in \cite{Ar72} and all groups $W\in\CW$ are torsion-free and quotient-divisible, the groups $d(W)$ are torsion-free and quotient-divisible.

(i) Since $d$ is a duality, for every $W\in \CW$ we have $\BBQ\End(d(W))\cong \BBQ\End(W)\cong \BBQ$. Then $\End(d(W))$ is isomorphic to a unital subring of $\BBQ$. It follows that  $\End(d(W))\ncong \BBZ$ if and only if there exists a prime $p$ such that $\End(d(W))$ is $p$-divisible. 

Suppose that there exists such a prime $p$. It follows that $d(W)$ is $p$-divisible. Since there exists a full free subgroup $F_W\leq W$ such that $W/F_W\cong (\BBQ/\BBZ)^{m-1}$, we can apply the statement 1) from Lemma \ref{prop-duality} to obtain a contradiction. Therefore, for every $W\in \CW$ we have $\End(d(W))\cong \BBZ$.

(ii) We have $\BBQ\Hom(V,K)=\BBQ\Hom(d(W),K)\cong \BBQ\Hom(d(K),W)=0$, hence $\Hom(V,K)=0$ since it is torsion-free.

(iii) and (iv) Let $U$ be a proper pure subgroup of $V=d(W)$. We apply $d$ to the exact sequence $$0\to U\overset{\iota}\to V\overset{\pi}\to V/U\to 0,$$ where $\iota$ is the inclusion map, and $\pi$ is the canonical surjection. By Remark \ref{rem-quasi-exact} it follows that we can assume w.l.o.g. that $\Ima(d(\pi))$ is a finite index subgroup of $\Ker(d(\iota))$ and that $\Ima(d(\iota))$ is a finite index subgroup of $d(U)$. But all proper torsion-free quotient subgroups of $d(V)$ are divisible, \cite[Theorem 2.1]{Go-W}. Since $d$ preserves the rank, it follows that $\Ima(d(\iota))$ is divisible,  and this implies that $d(U)$ is divisible. Then $U$ is free. 

Similarly, if $U$ is of rank $m-1$ then $d(V/U)$ is isomorphic to a (finite index) subgroup of the rank $1$ subgroup $\Ker(d(\iota))$. It follows that $d(V/U)\cong \BBZ$ since all rank $1$ subgroups of $d(V)$ are isomorphic to $\BBZ$. Then $V/U\cong \BBQ$. 
\end{proof}

\subsection{When the divisible part of $G/T(G)$ has rank $1$}

\begin{proposition}\label{G/T(G)-rank1}
Suppose that $G\in\CS^\star$ is a group such that the divisible part of $G/T(G)$ has rank $1$. Then for every positive integer $n$ there exists $H$ such that the pair $(G,H)$ verifies the conditions $\mathrm{(I)}$ and $(\mathrm{II})_n$. 

Consequently, if $G$ has the cancellation property then $\End_{W}(G)$ has the unit lifting property.
\end{proposition}

\begin{proof}
We have $G/T(G)=\overline{Q}\oplus \overline{M}$ such that $\overline{Q}\cong \BBQ$ and $ \overline{M}$ is a reduced torsion-free group. Let $T(G)\leq Q, M\leq G$ such that $Q/T(G)=\overline{Q}$ and $M/T(G)=\overline{M}$. Then $Q$ and $M$ are pure subgroups of $G$, \cite[Ex. S.3.26]{Ab-Gr-ex}. Moreover, $G/Q\cong \overline{M}\in\CS^\star$ is torsion-free, hence we can apply Proposition \ref{tf-W}.

Let $\mathcal{V}$ be an infinite class of groups constructed as in Proposition \ref{tf-W} such that for all $V\in \CV$ we have $r(V)=r_0(G)+1$ and $\Hom(V,\overline{M})=0$. Then $\Hom(G,V)=0$ for all $V\in \CV$.  %and $\Hom(G/T(G), V)=0$. 

Let $f:V\to G$ be a morphism, where $V\in \CV$. Suppose that $f(V)$ is not a torsion-group. Let $\pi:G\to G/T(G)$ and $\pi_{\overline{M}}:G/T(G)\to \overline{M}$ be the canonical projections. It follows that $\pi_{\overline{M}}\pi f=0,$ hence $\pi f(V)\leq \overline{Q}$ (note that $\overline{Q}$ is the unique direct complement of $\overline{M}$). Then the image $f(V)$ has the torsion-free rank $1$. 
%It follows that $f(V)/T(f(V))$ is a rank $1$ torsion-free quotient of $V$, hence it is divisible. 
Let $K$ be the kernel of $f$. Then the pure envelope $K_\star$ of $K$ is a proper pure subgroup of $V$, hence it is free, and it follows that $K_\star/K$ is finite. Since $V/K_\star\cong \BBQ$, it follows that $f(V)\cong V/K\cong B\oplus Q$, hence $G$ is not reduced, a contradiction. It follows that $f(V)$ is a reduced torsion group. But $V$ is quotient-divisible, and this implies that the reduced parts of all its quotients are bounded. Then $f(V)$ is bounded. Then there exists a positive integer $m$ such that $mf=0$. It follows that for all $V\in \CV$ the group $\Hom(V,G)$ is a torsion group. By Lemma \ref{suff} we conclude that for every positive integer $n$ there exists $H$ such that the pair $(G,H)$ verifies the conditions $\mathrm{(I)}$ and $(\mathrm{II})_n$.    
%
%
%Since $V$ is quotient-divisible and $f(V)$ is reduced, it follows that $f(V)$ is quotient-divisible. Since $F\cap f(V)$ is a full free subgroup of $f(V)$, it follows that $f(V)/F\cap f(V)$ is a sum of a divisible group and a finite group. Then there exists a positive integer $m$ such that $mf(V)\leq L$. If $\pi_1:L\to L_1$ and $\pi_Q:L_1\to L_1/Q=\overline{M}$ are the canonical projections, we obtain a morphism $\pi_Q\pi_1mf:V\to \overline{M}$. But all rank $1$ torsion-free quotients of $V$ are divisible, and $\overline{M}$ is reduced. It follows that $\pi_1mf(V)\subseteq Q$. It is easy to see that every quotient of $V$ which is of rank $1$ is a direct sum of a finite group and a divisible group. Then $\pi_1mf(V)$ is finite, and it follows that there exists a positive integer $n$ such that $nmf(V)=0$, hence $\Hom(V,G)$ is a torsion group.   
%
%By Lemma \ref{suff} it follows that for every positive integer $n$ there exists $H$ such that the pair $(G,H)$ verifies the conditions $\mathrm{(I)}$ and $(\mathrm{II})_n$, and the conclusion is a consequence of Proposition  \ref{fuchs-mixed}.
\end{proof}

\section{Self-small groups which verify Stelzer's theorem}

The main aim of this section is to identify classes $\CC\subseteq \CS^\star$ such that they satisfy the mixed version of Stelzer's Theorem, i.e. if $G\in \CC$ has the cancellation property then $\End_{\mathrm{W}}(G)$ has the unit lifting property. 

\subsection{A reduction lemma} As a first step we will prove a result which reduces our study to groups which have no mixed direct summands from the class $\CG$. We recall that a group $G\in \CS$ is in \textit{the class $\CG$} if $G/T(G)$ is divisible, \cite{AGW}. It was proved in \cite{Fi-W} that all groups from $\CG$ have the cancellation property. Moreover, for every $G\in\CG$ the Walk-endomorphism ring $\End_{\mathrm{W}}(G)$ is torsion-free and divisible, hence it has the unit lifting property. 

\begin{lemma}\label{fara-calG}
Let $G=L\oplus K\in \CS$ such that $L\in \CG$. Then $\End_{\mathrm{W}}(G)$ has the unit lifting property if and only if $\End_{\mathrm{W}}(K)$ has the unit lifting property. 
\end{lemma}

\begin{proof}
We write $\End(G)$ as a matrix ring $$\End(G)=\left(\begin{array}{cc}
\End(L) & \Hom(K,L) \\ \Hom(L,K) & \End(K)                                                    \end{array}\right).
$$
Since $G$ is self-small, it follows that if $X,Y\in\{L,K\}$ then $T(\Hom(X,Y))=\Hom(X,T(Y))$. Moreover, from the exact sequence 
$$0\to T(Y)\to Y\to Y/T(Y)\to 0,$$ we obtain an injective map 
$$\Hom(X,Y)/\Hom(X,T(Y))\hookrightarrow \Hom(X,Y/T(Y))\hookrightarrow \Hom(X/T(X),Y/T(Y)).$$ 
The image of this map is pure, \cite[Lemma 2.6]{Br-Warf}. 

In the matrix representation 
$$\End_{\mathrm{W}}(G)=\left(\begin{array}{cc}
\End_{\mathrm{W}}(L) & \Hom(K,L)/\Hom(K,T(L)) \\ \Hom(L,K)/\Hom(L,T(K)) & \End_{\mathrm{W}}(K)                                                    \end{array}\right),
$$
the additive groups  $$\End_{\mathrm{W}}(L),\  \Hom(K,L)/\Hom(K,T(L)), \textrm{ and }\Hom(L,K)/\Hom(L,T(K))$$ are torsion-free and divisible. 
Hence, for every positive integer $n$ the factor ring $\End_{\mathrm{W}}(G)/n\End_{\mathrm{W}}(G)$ can be identified with $\End_{\mathrm{W}}(K)/n\End_{\mathrm{W}}(K)$. The conclusion is now obvious.
\end{proof}

Consequently, in order to find classes of groups which satisfy the mixed version of Stelzer's theorem, it is enough to identify self-small mixed groups which are direct sums of a group from $\CG$ and a group which verify the hypothesis of Proposition \ref{fuchs-mixed}. 

\subsection{Self-small groups as pushouts}
We will describe (up to a finite summand) the groups in $\CS^\star$ in a manner which is similar, but not identical, to that presented in \cite[Proposition 3.2]{Br-S}, as pushouts of quotient-divisible groups and torsion-free groups.

\begin{lemma}\label{pushout}
Let $G\in \CS^\star$ be of torsion-free rank $k$. There exists a direct decomposition $G=V\oplus G'$ such that $V$is finite and $G'$ can be embedded in a pushout diagram 
$$\xymatrix{0\ar[r] & F \ar[r]\ar[d] &K\ar[r]\ar[d] & K/F\ar[r]\ar@{=}[d] & 0\\
0\ar[r] & L \ar[r] &G'\ar[r] & K/F\ar[r] & 0
}$$ such that 
\begin{enumerate}[{\rm (1)}]
\item $F$ is a full free subgroup $F\leq G'$, and $G'/F$ is $p$-divisible for all primes $p$ with $T_p(G')\neq 0$;
\item $L$ is quotient-divisible such that $T(G')\leq L$ and $L/F$ is divisible; 
\item $K$ is torsion-free and $K/F$ is reduced;
\item for every prime $p$ we have $(G'/L)_p= 0$ or $T_p(G')= 0$.
\end{enumerate}
\end{lemma}

\begin{proof}
Let $F_0$ be a full free subgroup of $G$. It follows that the set $$U=\{p\in \mathbb{P}\mid (G/F_0)_p\neq p(G/F_0)_p \textrm{ and } T_p(G)\neq 0\}$$ is finite. We consider the subgroup $V=\oplus_{p\in U}T_p(G)$ and we fix a direct decomposition $G=V\oplus G'$. Then $F=F_0\cap G'$ is a full free subgroup of $G'$, and $G'/F$ is isomorphic to a finite index subgroup of $G/F_0$. Suppose that there exists a prime $p$ such that $G'/F$ is not $p$-divisible and $T_p(G')\neq 0$. Then $G/F_0$ is not $p$-divisible and $T_p(G)\neq 0$, hence $p\in U$. This is not possible since $T_p(G')=0$ for all $p\in U$. It follows that for every prime $p$ such that $T_p(G')\neq 0$ the group $G'/F$ is $p$-divisible. 

We write $G'/F=\widehat{L}\oplus \widehat{K}$ such that $\widehat{L}$ is divisible and $\widehat{K}$ is reduced.  We denote by $L$ and $K$ the subgroups of $G'$ such that $F\leq L\cap K$, $L/F=\widehat{L}$, and $K/F=\widehat{K}$.     
We observe that $K$ is torsion-free and $T(G')\leq L$. Moreover, for every prime $p$ we have $(G'/L)_p= 0$ or $T_p(G')= 0$. 

In order to complete the proof, it is enough to apply \cite[Lemma 3.1]{Br-S} to conclude that all these data can be included in a pushout diagram.  
\end{proof}

\begin{proposition}\label{cazuri}
Let $G\in \CS^\star$ be a group which can be embedded in a pushout diagram as in Lemma {\rm \ref{pushout}}. Moreover, we fix a decomposition $L=F_1\oplus L_1$ such that  $F_1$ is free and $L_1$ has no free direct summands. 
\begin{enumerate}[{\rm a)}]
\item If $K$ is free then $G$ is quotient-divisible.
\item If $L$ is free then $G$ is torsion-free. 
\item If $L_1\in \CG$ then $G=L_1\oplus K_1$, where $K_1$ is torsion-free.
\item If $r(F_1)\leq 1$ then for every positive integer $n$ there exists $H$ such that the pair $(G,H)$ verifies the conditions $\mathrm{(I)}$ and $(\mathrm{II})_n$.  
\item If $r_0(L_1)=1$ then one of the following properties holds true:
			\begin{enumerate}[{\rm (i)}] \item $G/T(G)$ is reduced, or 
			\item $G=L_1\oplus K_1$, where $L_1\in \CG$ and $K_1$ is torsion-free.
			\end{enumerate}
\item If $r_0(L_1)=2$ then one of the following properties holds true:
			\begin{enumerate}[{\rm (i)}] 
			\item $G/T(G)$ is reduced;
			\item for every positive integer $n$ there exists $H$ such that the pair $(G,H)$ verifies the conditions $\mathrm{(I)}$ and $(\mathrm{II})_n$;
			\item $G=L_1\oplus K_1$, where $L_1\in \CG$ and $K_1$ is torsion-free.
			\end{enumerate}
\end{enumerate}    
\end{proposition}

\begin{proof}
a) This follows from the isomorphism $G/K\cong L/F$.

b) We have $T(G)\leq L$. Since $L$ is free, it follows that $T(G)=0$.

c) For every prime $p$ we have $(G/L)_p= 0$ or $T_p(G)= 0$. This implies that $\Ext(G/L,T(G))=0$. 

Moreover, $T(G)\leq L_1$ and $L_1/T(G)$ is divisible. Applying the covariant functor $\Ext(G/L,-)$ to the exact sequence $0\to T(G)\to L_1\to L_1/T(G)\to 0$, we obtain $\Ext(G/L,L_1)=0$. 
Therefore, in the exact sequence $$\Ext(G/L,L_1)\to \Ext(G/L_1,L_1)\to \Ext(L/L_1,L_1)$$ the first group is $0$. But $\Ext(L/L_1,L_1)=0$ because the group $L/L_1$ is free, hence $\Ext(G/L_1,L_1)=0$. This implies that $L_1$ is a direct summand of $G$. 
Then $G=L_1\oplus K_1$, where $K_1$ is reduced and torsion-free. 

d) We will prove that $G$ verifies the hypothesis of Proposition \ref{fuchs-mixed}.

We have to study the cases when $L$ has no free direct summands or when $L=\langle x\rangle \oplus L_1$, where $L_1$ has no free direct summands.
We will present the proof for the second case.  The proof for the first case can be done in a similar way. 

Let $\CW$ be a class associated to $L_1$ as in Lemma \ref{qd-H} such that all groups $W\in\CW$ have the rank $k+1$. Using Lemma \ref{suff} we conclude that it is enough to prove that for every $W\in\mathcal{W}$ we have $\Hom(G,W)=0$ and $\Hom(W,G)$ is a torsion group.

Let $f:W\to G$ be a morphism. 
The rank of $W$ is $k+1$, hence $\Ker(f)\neq 0$. Then $\Ima(f)$ is a direct sum of a divisible group and a finite group. But $G$ is reduced, hence $\Ima(f)$ is finite.

Suppose that there exists $f:G\to W$ a nonzero morphism. Since $f_{L_1}=0$, we obtain a morphism $\overline{f}:G/L_1\to W$, where the torsion-free rank of $G/L_1$ is $1$. Suppose that $\overline{f}\neq 0$. Then $\Ima(\overline{f})$ is a rank $1$ subgroup of $W$, hence $\Ima(\overline{f})\cong \BBZ$. This implies that $G$ has a direct summand isomorphic to $\BBZ$, a contradiction.

e) Suppose that $G/T(G)$ is not reduced. It follows that $L_1/T(G)$ is not reduced. 
Since the torsion-free rank of $L_1$ is $1$ then $L_1/T(G)\cong\BBQ$, hence $L_1\in \CG$. The conclusion follows from the case c).  

f) As in the previous case, we can assume that $G/T(G)$ is not reduced.
If $L_1/T(G)$ is divisible, we can apply the case c).
If $L_1/T(G)$ is not divisible, it follows that the divisible part of $G/T(G)$ has rank $1$, and the conclusion follows from Proposition \ref{G/T(G)-rank1}.
\end{proof}

\begin{corollary}\label{pc}
Suppose that $G\in \CS^\star$ has the cancellation property, and that it has a decomposition as in Lemma \ref{pushout} such that $G'$ verifies one of the hypotheses listed in Proposition \ref{cazuri}. Then $\End_{\rmW}(G)$ has the unit lifting property.
\end{corollary}

\begin{proof}
This follows from Lemma \ref{fara-calG} and Proposition \ref{fuchs-mixed}.
\end{proof}

\subsection{Groups of torsion-free rank at most 4} We are ready to prove that Stelzer's Theorem is true for self-small groups of torsion-free rank at most $4$. A proof for groups of torsion-free rank $1$ is presented in \cite[Proposition 3.5 and Corollary 3.6]{ABVW}.

\begin{theorem}\label{main-thm}
Let $G\in \CS^*$ be a group of torsion-free rank at most $4$. If $G$ has the cancellation property then $\End_{\rmW}(G)$ has the unit lifting property. 
\end{theorem}

\begin{proof} Consider a pushout diagram as in Lemma \ref{pushout}, and write $L=F_1\oplus L_1$ such that $F_1$ is free and $L_1$ has no free direct summands. It is easy to see that if $r_0(F_1)\geq 2$ then $r_0(L_1)\leq 2$. Hence we can apply Corollary \ref{pc}.
\end{proof}

\begin{corollary}\label{cor-main}
Suppose that $G$ is a self-small group of torsion-free rank at most $4$. If $G$ has the cancellation property then $G=F\oplus H$ such that $F$ is a free group and $\End_{\rmW}(H)$ has the unit lifting property.  
\end{corollary}

\begin{proof}
We write $G=F\oplus K\oplus Q$, where $F$ is free, $Q$ is divisible and $K\in \CS^\star$. Since $F$ and $Q$ have the cancellation property, it follows that $G$ has the cancellation property if and only if $K$ has the cancellation property. We recall that $Q$ is torsion-free, \cite{Ar-Mu}. If we denote $H=K\oplus Q$, the conclusion is obvious.     
\end{proof}

We recall from \cite{Br-qd} that a group $G\in \CS$ is \textit{strongly indecomposable} of $\BBQ\End_{\mathrm{W}}(G)$ has no non-trivial idempotents. From \cite[Theorem 2.2]{ABVW} and \cite[Proposition 3.1]{Br-2020} we obtain the following

\begin{corollary}
A strongly indecomposable self-small reduced group $G$ of torsion-free rank at most $4$ has the cancellation property if and only if there exists a finite group $B$ such that $G\cong B\oplus \BBZ$ or one is in the stable range of $\End_{\mathrm{W}}(G)$.
\end{corollary}

\end{document}